\documentclass{amsart}
\usepackage{amssymb,pgf,tikz}
\newcommand{\Sym}{\mathop{\mathrm{Sym}}}
\newcommand{\Alt}{\mathop{\mathrm{Alt}}}
\newtheorem{theorem}{Theorem}[section]
\newtheorem{corollary}[theorem]{Corollary}
\newtheorem{proposition}[theorem]{Proposition}
\newtheorem{lemma}[theorem]{Lemma}
\newtheorem{remark}[theorem]{Remark}
\newtheorem{question}[theorem]{Question}

\newtheorem{notation}[theorem]{Notation}

\renewcommand{\wr}{\mathop{\mathrm{wr}}}

\author[D.~Bubboloni]{Daniela Bubboloni}
\address{Daniela Bubboloni, Dipartimento di Scienze per l'Economia e l'Impresa,\newline
University of Firenze, \newline Via delle Pandette 9-D6,  50127 Firenze, Italy}
\email{daniela.bubboloni@unifi.it}

\author[C. E. Praeger]{Cheryl E. Praeger}
\address{Cheryl E. Praeger, Centre for Mathematics of Symmetry and Computation,\newline
School of Mathematics and Statistics,\newline
The University of Western Australia,\newline
 Crawley, WA 6009, Australia}\email{Cheryl.Praeger@uwa.edu.au}

\author[P. Spiga]{Pablo Spiga}
\address{Pablo Spiga,
Dipartimento di Matematica e Applicazioni,\newline
 University of Milano-Bicocca,\newline Via Cozzi 53, 20125 Milano, Italy
}
\email{pablo.spiga@unimib.it}

\thanks{
The first author is supported by the MIUR project ``Teoria dei gruppi
ed applicazioni (2009)".\\ The second author is supported by Australian Research
Council Federation Fellowship FF0776186; she is also affiliated with King Abdulazziz
University, Jeddah, Saudi Arabia.}

\subjclass[2000]{20B30}
\keywords{symmetric groups; conjugacy classes; normal coverings; generation of groups}

\begin{document}


\title[Normal coverings of symmetric groups]{Normal coverings and pairwise generation of finite alternating
and symmetric groups}

\begin{abstract}
The normal covering number $\gamma(G)$ of a finite, non-cyclic group $G$ is the least number
of proper subgroups  such that each element of $G$  lies in some conjugate
of one of these subgroups. We prove that there is a positive constant $c$ such that,
for $G$ a symmetric group $\Sym(n)$ or an alternating group $\Alt(n)$, $\gamma(G)\geq cn$.
This improves results of the first two authors who had earlier proved that
 $a\varphi(n)\leq\gamma(G)\leq 2n/3,$ for some positive constant $a$, where $\varphi$
 is the Euler totient function.
Bounds are also obtained for the maximum size  $\kappa(G)$ of a set $X$ of
conjugacy classes of $G=\Sym(n)$ or $\Alt(n)$ such that any pair of
elements from distinct classes in $X$ generates $G$, namely
$cn\leq \kappa(G)\leq 2n/3$.
\end{abstract}

\maketitle

\section{Introduction}\label{sec:intro}

Let $G$ be the symmetric group $\Sym(n)$ or the alternating group $\Alt(n)$ for $n\in\mathbb{N}.$ If
$H_1,\ldots,H_r$ are pairwise non-conjugate proper subgroups of $G$
with
$$
G=\bigcup_ {i=1}^r\bigcup_{g\in G}H_i^g,
$$
we say that $\Delta=\{H_i^g\,|\,1\leq i\leq r, g\in G\}$ is a {\em
  normal covering} of $G$ and that $\delta=\{H_1,\dots,H_r\}$ is a
{\em basic set} for $G$ which generates $\Delta.$ We call the elements
of $\Delta$ the {\em components}  and
the elements of $\delta$ the {\em basic components} of the normal covering.
In \cite{BP} the first two authors  introduced  the {\em normal
  covering number} of $G$ as
$$
\gamma(G):=\min\{|\delta|: \delta\  \hbox{is a basic set of} \ G\},
$$
and obtained lower and  upper estimates for $\gamma (G)$ as a
function of $n.$ They  found~\cite[Theorems~A and~B]{BP} that
$a\varphi(n)\leq \gamma (G)\leq
bn$  where $\varphi(n)$ is the
Euler totient function and $a, b$ are positive real constants depending on
whether $G$ is alternating or symmetric and whether $n$ is even or odd.
Note that $\gamma(\Sym(n))$, $\gamma(\Alt(n))$ are defined only for
$n\geq 3$ and $n\geq 4,$ respectively, since normal coverings do not exist for cyclic groups.
Our aim in this paper is to improve significantly the lower bound in \cite{BP}
by establishing a lower bound linear in $n$. 

In addition we investigate a second parameter  $\kappa(G)$, of a finite group $G$, introduced
very recently by Britnell and Mar\'oti~\cite{BM}: $\kappa(G)$ is the maximum size (possibly zero) of a set $X$
of conjugacy classes of $G$ such that any pair of elements from distinct classes in $X$
generates $G$. Such a set $X$ is called an \emph{independent set} of classes. From this
definition, it is clear  that
\begin{equation}\label{kappa-gamma}
\kappa(G)\leq \gamma(G)
\end{equation}
for every finite non-cyclic group $G$.

Consider now the groups $G=\Sym(n)$ or $G=\Alt(n).$
It is known that $\kappa(\Alt(n)\geq 2$ for all $n\geq 5$ (see Proposition 7.4 in \cite{GM}) and the same inequality holds trivially for $3\leq n\leq 4.$ In Proposition \ref{kappasym}, we show that also
$\kappa(\Sym(n))\geq 2$ for $n\geq 2,\ n\neq 6.$ The case $\Sym(6)$ is a genuine exception since $\kappa(\Sym(6))=0.$
Note also that, since $\gamma(\Sym(6))=2$, the functions $\kappa$ and $\gamma$ are different.

Our main
result says that, in the non-cyclic case, both parameters $\gamma(G)$ and $\kappa(G)$ grow linearly with the degree
$n$ of $G$ in the following sense.

\begin{theorem}\label{thm3}
There exists a positive real constant 
$c$ such that for  $G=\Sym(n)$ with $n\geq 3, n\neq 6$, or $G=\Alt(n)$ with $n\geq4$, we have
$$
cn \leq \kappa(G)\leq \gamma(G) \leq \frac{2}{3} n.
$$
\end{theorem}

The upper bound in Theorem~\ref{thm3} is proved in \cite[Theorem A and Table 3]{BP};
our task here is to prove the lower bound. Our proof give a less than optimal value for $c$ since we choose to avoid the classification of the finite simple groups. Thus this
result raises naturally a question about the constants. We note that
$\gamma(\Sym(3))=2$ so the upper bound is sharp. Also $\kappa(\Sym(4))=\kappa(\Alt(4))=2$, so the constant $c\leq 1/2$.
\begin{question}
What is the best positive real constant $c$ such that $\kappa(G)\geq cn$, for $G=\Sym(n)$ or $\Alt(n)$ and all $n\in\mathbb{N}$ with $n\geq 3$ and $(n,G)\neq (6, \Sym(6))$? 
\end{question}
Our proof strategy is to demonstrate that  any normal covering consisting of maximal subgroups must
contain sufficiently many maximal intransitive subgroups, and that any maximal independent set of conjugacy
classes must contain sufficiently many classes from a specified set in which every maximal subgroup,
containing an element from one of these classes, is intransitive. A more explicit lower bound is
given in Theorem~\ref{main1} below.
To state this result we introduce the two real numbers:
\begin{equation}\label{eq1}
\alpha:=\prod_{p\,\mathrm{prime}} \left(1-\frac{2}{p^{2}}\right)\quad\textrm{and}\quad
\beta:=\prod_{p\,\mathrm{prime}} \left(1-\frac{3p-2}{p^{3}}
\right).
\end{equation}
\begin{theorem}\label{main1}
For every $\varepsilon\in\mathbb{R}$ with
$\varepsilon>0$, there exists  $n_\varepsilon\in\mathbb{N}$ such that,
for every $n\in \mathbb{N}$ with $n\geq n_\varepsilon$, we have
\[
\frac{\kappa(\Sym(n))}{n}\geq \left\{
\begin{array}{ccl}
\displaystyle{\frac{\alpha}{12}}-\varepsilon&&\textrm{if }n \textrm{ is even},\\
&&\\
\displaystyle{\frac{\beta}{120}}-\varepsilon&&\textrm{if }n \textrm{ is odd,}\\
\end{array}\right.\]
\[
\frac{\kappa(\Alt(n))}{n}\geq \left\{
\begin{array}{ccl}\displaystyle{\frac{\beta}{150}}-\varepsilon&&\textrm{if }n
  \textrm{ is even},\\
&&\\
\displaystyle{\frac{\alpha}{18}}-\varepsilon&&\textrm{if }n \textrm{ is odd}.\\
\end{array}\right.
\]
Moreover, we have $\alpha>0.32263$ and $\beta>0.28665$.
\end{theorem}

In Remark~\ref{computational} we give a computational version of
Theorem~\ref{main1}, with an explicit value for $n_\varepsilon$, and
we describe a general method for obtaining some estimates on $n_\varepsilon$
given a fixed $\varepsilon$.
In the proof of Theorem~\ref{main1} we follow and
deeply strengthen some ideas in~\cite{BP}, through a recent
number theoretic result~\cite[Theorem~$2$]{BLS}.
Theorem~\ref{thm3} follows  immediately from
Theorem~\ref{main1}.

\begin{remark}{\rm
Our results also give  insight into a further parameter $\mu(G)$
defined on finite groups $G$ by Blackburn in \cite{Bl}: the
quantity $\mu(G)$ is the maximum size of a subset $X$ of $G$ such
that for distinct $g, h \in X$, we have  $G=\langle g, h\rangle.$
It follows immediately from this definition that $\mu(G)\geq
\kappa(G)$, and on the other hand $\mu(G)$ is at most the covering
number $\sigma(G)$ (the minimum number of proper subgroups of $G$
whose union is equal to $G$). Thus $\kappa(G)\leq\mu(G)\leq\sigma(G)$
for all finite non-cyclic groups $G$. Moreover Blackburn~\cite{Bl}
asked whether $\sigma(G)/\mu(G)\rightarrow 1$ for
finite simple groups $G$ as the order of $G$ tends to infinity.

In~\cite{Bl} Blackburn used probabilistic methods to prove that
$\mu(\Sym(n))=2^{n-1}$ for  sufficiently large odd $n$, and
$\mu(\Alt(n))=2^{n-2}$ for sufficiently large $n\equiv 2\pmod{4}$;
and in these cases the parameter $\mu(G)$ is equal to $\sigma(G)$
by a result of Mar\'oti~\cite[Theorem 1.1]{M}. To our knowledge
the value of $\mu(G)$, when $G=\Sym(n)$ with $n$ even or $G=\Alt(n)$
with $n\not\equiv 2\pmod{4}$, is  not known. It follows from
Theorem~\ref{thm3} that $\mu(G)$ grows at least linearly in $n$ in
these cases. We suspect that this is far from the truth in the light of Blackburn's question mentioned in the previous paragraph,
and Mar\'oti's results \cite[Theorems~3.2,~4.1,~4.2,~4.4]{M}, that
$\sigma(G)$ grows at least exponentially with $n$ (as long as  $G$ is not $\Alt(n)$ with
$n=(q^k-1)/(q-1)$ for some $q$ prime and $k\in\mathbb{N}$).
}
\end{remark}

In the rest of this introductory section, we set some notation and definitions that we
will use throughout the rest of the paper and we give some insights in
how~\cite[Theorem~$2$]{BLS} is used to investigate $\gamma(G)$ and $\kappa(G)$.
From now on $G$ will denote always the group $\Sym(n),$ with $n\geq 3$ or the group $\Alt(n),$ with $n\geq 4,$ $n\in\mathbb{N}.$

\subsection{Partitions and permutations}
\label{part-perm}
Let $n,k\in \mathbb{N},$  with $n\geq k.$ A $k$-{\em
  composition} of $n$ is an ordered $k$-tuple
$(x_1,\dots,x_k)\in\mathbb{N}^{k}$ such that $n=\sum_{i=1}^{k}x_i.$ A
$k$-{\em partition} of $n$ is an unordered $k$-tuple
$[x_1,\dots,x_k]$, with $x_i\in\mathbb{N}$ for each $i\in
\{1,\dots,k\},$ such that $n=\sum_{i=1}^{k}x_i.$
The $x_i$s are called the {\em terms} of the composition or of the partition.
We denote with
\begin{eqnarray}
\label{Kappa}
\mathcal{K}_k(n)&:=&\{(x_1,x_2,\dots,x_{k})\in\mathbb{N}^{k}
:\  n=\sum_{i=1}^{k}x_i\},\\
\label{P}
\mathfrak{P}_k(n)&:=&\{[x_1,x_2,\dots,x_{k}]
:\ x_i\in\mathbb{N}\  \hbox{for}\  i=1,\dots,k,\  n=\sum_{i=1}^{k}x_i\}
\end{eqnarray}
respectively, the set of  $k$-compositions and the set of
$k$-partitions of $n.$ Moreover we call each element in  $\mathfrak{P}(n)=\bigcup_{k=1}^n\mathfrak{P}_k(n)$ a {\em partition} of $n.$

Given a permutation $\sigma\in \Sym(n)$ which decomposes as a
product of $k$ pairwise disjoint cycles of lengths $x_1,\dots, x_k,$ we
associate with $\sigma$ the $k$-partition
$\mathfrak{p}(\sigma)=[x_1,\dots, x_k]\in \mathfrak{P}(n)$ which we
call the {\em type} of $\sigma.$ Note that  the map
$$\mathfrak{p}:\Sym(n)\rightarrow \mathfrak{P}(n)$$ is
surjective, that is, each partition of $n$ may be viewed as the type of
some permutation. Note also that $\mathfrak{p}(\Alt(n))\subsetneqq
\mathfrak{P}(n)$, for every $n\geq 2.$

If $H\leq \Sym(n)$ we say that the type $ \mathfrak{p}\in
\mathfrak{P}(n)$ {\em belongs to} $H$ or that $H$ {\em contains} $ \mathfrak{p}$
 if $ \mathfrak{p}\in \mathfrak{p}(H),$ that is, $H$ contains a
 permutation of type $\mathfrak{p}.$

Observe that two permutations are conjugate in  $\Sym(n)$ if and only if they
have the same type, and the set of permutations of $\Alt(n)$ with the same type splits into at most two conjugacy
 classes of $\Alt(n)$. This conjugacy class splitting will not be a problem for our analysis.
Observe also that we can always replace a normal covering of $G=\Sym(n)$
or $\Alt(n)$ by one with the same
number of basic components in which each component is a maximal
subgroup of $G$. For this reason we can assume (and we will) that each
component is a maximal subgroup of $G$, that is to say, in
its action on $\{1,\ldots,n\}$, it is  a maximal intransitive, or
imprimitive, or primitive subgroup of $G$. Similarly, when seeking
maximal independent sets of conjugacy classes of $G$, we
are concerned with the sets of types belonging to maximal intransitive,
imprimitive, and primitive subgroups of $G$.\\


\subsection{Order of magnitude}\label{magnitude}  We introduce now the notion of order of magnitude, which we will use throughout the paper to get an immediate incisive relation between the sizes of various sets we deal with. This is a standard notion in computer science.

A subset  $D$ of $\mathbb{N},$ is called an {\it unbounded domain} of $\mathbb{N}$ if for each $k\in \mathbb{N},$ there exists $n_k\in D$ with $n_k\geq k.$
Given an unbounded domain $D\subseteq \mathbb{N}$ we  consider the set of real functions on $D$
$$
\mathcal{F}_D=\{f:D\rightarrow\mathbb{R} \}
$$
and the set of positive real functions on $D$
$$
\mathcal{F}_D^{> 0}=\{f:D\rightarrow\mathbb{R}^{> 0}\}.
$$
\begin{itemize}
\item[a)] For $f,g\in \mathcal{F}_D,$ we write $f(n)=o(g(n))$ if $\displaystyle{\lim_{n\rightarrow +\infty}\frac{f(n)}{g(n)}=0}$. If $f,g\in \mathcal{F}_D^{> 0}$ and $f(n)=o(g(n)),$ we say that
$f$ has {\it order of magnitude less than} $g$ or that $g$ has {\it order of magnitude greater than} $f$;

\item[b)] For $f,g\in \mathcal{F}_D^{> 0},$ we say that $f$ has the {\it same order of magnitude} as $g$  if there exist $a,b\in\mathbb{R}^{> 0}$ such that $$a\leq \frac{f(n)}{g(n)}\leq b, \qquad \text{for all}\ n\in D. $$
If the first inequality holds we say that $f$ has {\it order of magnitude at least} $g$; if the second inequality holds we say that $f$ has {\it order of magnitude at most} $g.$
\end{itemize}

Since the functions in $\mathcal{F}_D^{> 0}$ are positive, the inequalities in b) hold for all $n\in D$ if and only if they  hold (for possibly different positive $a, b$) for all $n\in D$ with $n\geq n_0,$ for some fixed $n_0$. Moreover it is immediate to check that to have the same order of magnitude is an equivalence relation on $\mathcal{F}_D^{> 0}$.

For  a finite number $D_1,\dots,D_m$ of unbounded domains of $\mathbb{N}$, if  $D=\cap_{i=1}^m D_i$ is also an unbounded domain of $\mathbb{N}$, then we may view the functions
$f_i\in \mathcal{F}_{D_i}$  as elements of $ \mathcal{F}_{D}$ and apply there the various notions of order of magnitude.

A relevant case  in which two functions have the same order of magnitude is the following.
	\begin{lemma}\label{o} Let  $f,g\in \mathcal{F}_D^{> 0}.$ If there exist $a,b\in\mathbb{R}^{> 0}$ and $\gamma_1, \gamma_2\in\mathcal{F}_D$ with both $\gamma_1(n)= o(g(n))$ and $\gamma_2(n)= o(g(n))$ such that  $$ag(n)+\gamma_1(n)\leq f(n)\leq bg(n)+\gamma_2(n), \qquad \text{for all}\ n\in D, $$ then $f$ has the same order of magnitude as $g.$
	\end{lemma}
	\begin{proof}  Since $g(n)>0,$ by the assumption we get
$$
a+\frac{\gamma_1(n)}{g(n)}\leq \frac{f(n)}{g(n)}\leq b+\frac{\gamma_2(n)}{g(n)}, \qquad \text{for all}\ n\in D.
$$
Fix an $\overline{\epsilon}$ with $0<\overline{\epsilon}<a.$ By assumption, $\displaystyle{\lim_{n\rightarrow +\infty}\frac{\gamma_1(n)}{g(n)}}=\displaystyle{\lim_{n\rightarrow +\infty}\frac{\gamma_2(n)}{g(n)}= 0},$
and hence there exists $\overline{n}\in\mathbb{N}$ such that
$$
\frac{\gamma_1(n)}{g(n)}>-\overline{\epsilon}\quad \mbox{and} \quad \frac{\gamma_2(n)}{g(n)}<\overline{\epsilon}
$$
for all $\  n\in D$ with $  n\geq \overline{n}.$ Thus
$$
0<a-\overline{\epsilon}\leq \frac{f(n)}{g(n)}\leq b+\overline{\epsilon},
$$
for all $n\in D $ with $n\geq \overline{n}.$
By the observation above, $f$ has the same order of magnitude as $g.$
	\end{proof}

\subsection{The idea of the proof of Theorem \ref{main1}}
\label{idea}
We need two more definitions before we can discuss the idea
of our proof.
For every $n,k \in\mathbb{N}$ with $n\geq k\geq 2,$ define
\begin{eqnarray}\label{eq2}
\mathfrak{T}_{k}(n):=\{[x_1,\dots, x_{k}]
\in\mathfrak{P}_{k}(n)&:&\exists\, i\in\{1,\dots,k\}\,  \hbox{with }\\\nonumber
&&  2\leq x_i< n/2,\,\gcd(x_i,\prod_{\substack{j=1\\j\neq i}}^{k}
x_j)=1\},
\end{eqnarray}
\begin{eqnarray}\label{eq3}
\hspace{-1.3cm}\mathcal{A}_{k-1}(n):=\{(x_1,x_2,\dots, x_{k})
\in\mathcal{K}_{k}(n)&:& \gcd(x_1,\prod_{j=2}^{k} x_j)=1\}.
\end{eqnarray}
The way $|\mathfrak{T}_{k}(n)|$ grows as a function of $n$ can
be easily derived from studying $|\mathcal{A}_{k-1}(n)|,$ because these two quantities have the same order of magnitude (see Theorem~\ref{BLS} and Corollary~\ref{Torder2}).

Moreover it was shown recently by two of the authors
with Luca~\cite[Theorem ~$2$]{BLS} that, for each $k\geq 3$,
$|\mathcal{A}_{k-1}(n)|$ has the same order of magnitude  as
$n^{k-1}$ while, anomalously, $|\mathcal{A}_{1}(n)|=\varphi(n)$ does not have the same order of magnitude as $n.$

Assume that $k$ is odd if $G=\Sym(n)$ and $n$ is even or if $G=\Alt(n)$
and $n$ is odd, and assume that $k$ is even if $G=\Sym(n)$ and $n$ is odd or if
$G=\Alt(n)$ and $n$ is even. Then for each type $\mathfrak{p}$ in
$\mathfrak{T}_k(n)$, every basic set of $G$ must contain a
component containing $\mathfrak{p}.$

 The types in $\mathfrak{T}_{k}(n)$ turn out to be very important for many reasons.
Firstly, no primitive proper subgroup of $G$ can contain elements from
$\mathfrak{T}_k(n)$ (see Lemma~\ref{evenodd}), so to cover these types
we need maximal intransitive subgroups or maximal imprimitive subgroups.
Secondly, choosing $k\in \{3,4\}$, it turns out that very few  types in
$\mathfrak{T}_{k}(n)$
belong to imprimitive subgroups, in the sense that their number has
 an order of magnitude less than $n^{k-1}$ (see Lemma \ref{lemma2}).
Finally for $k\in \{3,4\},$ each intransitive
subgroup can contain types from $\mathfrak{T}_{k}(n)$ within an order of
magnitude at most $n^{k-2}$ (see Lemma~\ref{lemma6}). It follows that the number of
intransitive basic components in any
normal covering of $G$ must be at least linear in $n$:
reaching in this way the required lower bound for
$\gamma(G)$. We require a slightly more refined argument to get a
lower bound for $\kappa(G)$.

Observe that our choice $k\in\{3,4\}$ is in a sense arbitrary. It is made on the one hand to guarantee control of the imprimitive and intransitive components (see Lemmas \ref{lemma2} and \ref{lemma6}) and on the other hand
for optimizing the sizes of the sets $\mathcal{A}_{k-1}(n)$ (see~\cite[Introduction]{BLS} for a discussion of the asymptotics of $|\mathcal{A}_{k-1}(n)|$), in the sense that this
choice gives the best constants in
Theorem~\ref{main1}.

We point out that the deep number theoretic results in~\cite{BLS},
estimating the size of the sets  $\mathcal{A}_{k-1}(n)$, are essential
for our purposes. Namely we cannot avoid a fine computation based of the entire set $\mathcal{A}_{k-1}(n),$ for $k\geq 3:$
in~\cite{BP} by working only on $k=2$ and a
special subset of $3$-partitions from $\mathcal{A}_{2}(n)$, it was shown
that an order of magnitude $\varphi(n)$ of types in these subsets must be contained exclusively  in intransitive components. This led to a
lower bound  for $\gamma(G)$ only of the same order of magnitude as $\varphi(n)$.

Note also that the distinction between the odd and even case in Theorem~\ref{main1}
arises because for $n$ even the $3$-partitions  of $n$ do not belong
to the alternating group while the $4$-partitions do, and for $n$ odd
the $4$-partitions do not belong to the alternating group while the
$3$-partitions do.


\section{Compositions, partitions and a number theory result}\label{2}


To develop our method we need to count various sets of
$k$-partitions. However, in general, it is easier to count sets of
$k$-compositions: for instance the order $K_k(n)$ of the set $\mathcal{K}_k(n)$ of the $k$-compositions,
 defined in \eqref{Kappa}, is given by
\begin{equation}\label{K}
K_k(n)=\binom{n-1}{k-1}=\frac{n^{k-1}}{(k-1)!}+o(n^{k-1}),
\end{equation}
(see ~\cite{Gould}) so that, by Lemma \ref{o}, $K_k(n)$ has the same order of magnitude as $n^{k-1}.$

Luckily there is a natural link between the set $\mathcal{K}_k(n)$ and
that of the $k$-partitions $\mathfrak{P}_k(n)$ defined in \eqref{P}, given by the map
$$\chi:\mathcal{K}_k(n)\to\mathfrak{P}_k(n),$$
where $\chi\left((x_1,\ldots,x_k)\right)=[x_1,\ldots,x_k]$. This map is obviously surjective and
 for each $\mathfrak{p}=[x_1,\dots,x_k]\in \mathfrak{P}_k(n)$ there
 exist at most $k!$ compositions in $\mathcal{K}_k(n)$ mapped by
 $\chi$ into $\mathfrak{p}.$
So if $\mathfrak{P}\subseteq
 \mathfrak{P}_k(n)$ and if $X\subseteq \mathcal{K}_k(n)$, then we have
 the following bounds

\begin{equation}
\label{part-comp}\frac{|\chi^{-1}(\mathfrak{P})|}{k!}\leq
   |\mathfrak{P}|\leq |\chi^{-1}(\mathfrak{P})|\quad\textrm{and}\quad
\frac{|X|}{k!}\leq\frac{ |\chi^{-1}(\chi(X))|}{k!}\leq |\chi(X)|.
\end{equation}
In particular the size of a set $\mathfrak{P}$ of $k$-partitions has the same order of magnitude as the size of the corresponding set $\chi^{-1}(\mathfrak{P})$ of $k$-compositions and the order of magnitude of $\mathfrak{P}_k(n)$ is $n^{k-1}.$

Recalling the definition of $\mathcal{A}_{k-1}(n)$ in~\eqref{eq3}, we introduce a notation for the corresponding $k$-partitions.
\begin{notation}\label{e}
For each $n,k \in\mathbb{N}$ with $n\geq k\geq 2,$ put
\begin{equation*}
 \mathfrak{A}_k(n):=\chi(\mathcal{A}_{k-1}(n))\quad\textrm{and}\quad
  \hat{A}_k(n):=|\mathfrak{A}_k(n)|.
\end{equation*}
\end{notation}

Clearly $\mathfrak{A}_k(n)$ is the set of $k$-partitions of $n$ in which one
term is coprime to the others. Observe explicitly that
\begin{equation}\label{mathfrak-A-2}
 \mathfrak{A}_2(n)=\{[x,n-x]\ :\ 1\leq x<n/2,\ \gcd(x,n)=1\}, \quad \hat{A}_2(n)=\varphi(n)/2 .
\end{equation}
 Moreover the set
 $\mathfrak{T}_k(n)$ defined in \eqref{eq2} is a subset of
 $\mathfrak{A}_k(n).$ Note that  $\hat{A}_k(n),$ as well as $|\mathcal{A}_{k-1}(n)|,$ are  positive functions defined on the unbounded domain $D_k=\{n\in \mathbb{N}\,:\,n\geq k\},$ because the set $\mathcal{A}_{k-1}(n)$ contains at least a $k$-composition with one term equal to $1.$ On the contrary  we cannot exclude that $\mathfrak{T}_k(n)=\varnothing$ for certain $n,k \in\mathbb{N};$ anyway we will show in Corollary \ref{Torder2} that also $|\mathfrak{T}_k(n)|$ is a positive function on an suitable unbounded domain of $\mathbb{N}$ depending on $k.$

From the inequalities \eqref{part-comp} we have that the order of magnitude of $ \hat{A}_k(n)$ is the same as $|\mathcal{A}_{k-1}(n)|$  for all $k\geq 2$ and so \cite[Theorem $2$]{BLS}  give the following important result.

\begin{theorem}\label{BLS} Let $n,k \in\mathbb{N}$ with $n\geq k\geq 3.$ Then $|\mathcal{A}_{k-1}(n)|$  and $ \hat{A}_k(n)$ have both order of magnitude $n^{k-1}.$ Moreover $|\mathcal{A}_{1}(n)|$ and $ \hat{A}_2(n)$ have both order of magnitude $\varphi(n).$
\end{theorem}


When $k\in\{3,4\}$ we want to find an explicit lower
bound for $\hat{A}_k(n)$ using some
notation, in part taken directly from~\cite{BLS}.
\begin{notation}\label{constants} Let $n\in\mathbb{N}.$ Write
\begin{equation*}
\begin{aligned}
&C_2:=\prod_{p\,\mathrm{prime}} \left(1-\frac{2}{p^{2}}\right),\qquad  &&C_3:=\prod_{p\,\mathrm{prime}} \left(1-\frac{3p-3}{p^{3}}\right),\\
&\alpha':=\prod_{p\,\mathrm{prime}}\left(1-\frac{1}{p^3-3p+3}\right),\qquad
  &&\beta:=C_3\alpha',\\
 &e_3(n):=-\frac{2+e}{12\sqrt{\pi }}(e^2\log n)^{2}n^{},&\quad &
  e_4(n):=-\frac{2+e}{24\sqrt{6\pi }}(e^2\log n)^{3}n^{2},\\
&f_2(n):=\prod_{\substack{p\mid n\\p\,\mathrm{prime}}} \left(1+\frac{1}{p^2-2}\right),\qquad &&f_3(n):=\prod_{\substack{p\mid n\\p\,\mathrm{prime}}} \left(1-\frac{1}{p^3-3p+3}\right).
 \end{aligned}
 \end{equation*}
 \end{notation}

Observe
that $$\left(1-\frac{3p-3}{p^3}\right)\left(1-\frac{1}{p^3-3p+3}\right)=1-\frac{3p-2}{p^3}$$
and so the definition of $\beta$ given here is consistent with the definition
given in~\eqref{eq1}. Observe also that $C_2$ equals $\alpha$ as
defined in~\eqref{eq1}.
The constants $C_i$ were defined and used in~\cite{BLS}, and it is convenient to use them here to make clear various parallels between the numbers $\hat{A}_3(n)$ and $\hat{A}_4(n)$. However to make the statement of
Theorem~\ref{main1} more elegant we renamed $C_2$ as
$\alpha$. To avoid any possible confusion, we use the label
$\alpha$ only in the statement of Theorem~\ref{main1}.

As a particular case of~\cite[Theorem~$2$]{BLS}, we obtain the
following theorem.

\begin{theorem}\label{coprimepart}\begin{description} \item[(i)] For all $n\in\mathbb{N},$ with $n\geq 3,$ we have
$$\hat{A}_3(n)\geq
\frac{C_2\,f_2(n)}{12}n^{2}+e_3(n)
$$
and $e_3(n)=o(n^2)$
\item[(ii)] For all $n\in\mathbb{N},$ with $n\geq 4,$ we have
$$\hat{A}_4(n)\geq
\frac{C_3\,f_3(n)}{144}n^{3}+e_4(n)$$
and $e_4(n)=o(n^3).$
\end{description}
\end{theorem}
\begin{proof}
Theorem~$2$ in~\cite{BLS} applied with $k=3$ and $k=4$ gives
\begin{eqnarray*}
|\mathcal{A}_2(n)|&\geq& \frac {C_2\,f_2(n)}{2}n^{2}-\frac{2+e}{2\sqrt{\pi
}}(e^2\log n)^{2}n^{},\\
|\mathcal{A}_3(n)|&\geq& \frac {C_3\,f_3(n)}{6}n^{3}-\frac{2+e}{\sqrt{6\pi
}}(e^2\log n)^{3}n^{2}.
\end{eqnarray*}
Now the proof follows from Notation \ref{e}, \ref{constants} and inequalities~\eqref{part-comp}.
\end{proof}

In order to apply Theorem~\ref{coprimepart} efficiently and obtain a
lower bound for $\kappa(\Sym(n))$ and $\kappa(\Alt(n))$, we need to
get precise estimates for $C_2$, $C_3$, $f_2$ and $f_3$. We do this in
the next lemma.

\begin{lemma}\label{f_23}
\begin{description}
 Let $n\in\mathbb{N}.$
\hspace{-0.01cm}\item[(i)] If $n$ is even, then $f_2(n)\geq \displaystyle{\frac{3}{2}}$. If $n\geq 3$ is
  odd, then $f_2(n)>1$.
\item[(ii)] If $n\geq 4$ is even, then $f_3(n)>\alpha'.$ If $n\geq 3$ is odd, then
  $f_3(n)>\displaystyle{\frac{5}{4}}\alpha'$.
  \vspace{0.2cm}
\item[(iii)]$C_2>0.32263$, $C_3>0.38159$ and $\beta>0.28665$.
\end{description}
\end{lemma}
\begin{proof}
Parts~(i) and~(ii) follow immediately from the definitions of $f_2(n)$
and $f_3(n).$ Part~(iii): the values for $C_2$ and $C_3$ are
in~\cite[Table~$1$]{BLS}. To estimate $\beta$ we use~\eqref{eq1} and
we let $p_i$ be the $i^{\textrm{th}}$ prime number, that is, $p_1=2$, $p_2=3$, etc.
For every $k\in \mathbb{N},$ write
 $u_k:=\prod_{i=1}^k\left(1-\frac{3p_i-2}{p_i^3}\right).$ So
 $\beta=u_k \prod_{i>k}\left(1-\frac{3p_i-2}{p_i^3}\right)$ and $\log(\beta)$ is equal to
$$
\log(u_k)-\sum_{i>k}\log\left(\left(1-\frac{3p_i-2}{p_i^3}\right)^{-1}\right)=\log(u_k)-\sum_{i>k}\log\left(1+\frac{3p_i-2}{p_i^3-3p_i+2}\right).
$$
 On the other hand, as $\log(1+x)\leq x$, the sum $\sum_{i>k}\log\left(1+\frac{3p_i-2}{p_i^3-3p_i+2}\right)$ is at most

\begin{eqnarray*}
\sum_{i>k}\frac{3p_i-2}{p_i^3-3p_i+2}
&\leq& \frac{3p_{k+1}-2}{p_{k+1}^3-3p_{k+1}+2}+\int_{p_{k+1}}^{+\infty} \frac{3x-2}{(x-1)^2(x+2)}dx\\
&=&\frac{3p_{k+1}-2}{p_{k+1}^3-3p_{k+1}+2}-\frac{8}{9}\log\left(\frac{p_{k+1}-1}{p_{k+1}+2}\right) +\frac{1}{3(p_{k+1}-1)}.
\end{eqnarray*}
Thus
\[
\log(\beta)\geq \log(u_k)-\frac{3p_{k+1}-2}{p_{k+1}^3-3p_{k+1}+2}+\frac{8}{9}\log\left(\frac{p_{k+1}-1}{p_{k+1}+2}\right) -\frac{1}{3(p_{k+1}-1)}
\]
and so $$\beta\geq
u_k\left(\frac{p_{k+1}-1}{p_{k+1}+2}\right)^{8/9}\exp\left(-\frac{p_{k+1}^2+10p_{k+1}-8}{3(p_{k+1}^3-3p_{k+1}+2)}\right).$$
Now the estimate on $\beta$ follows with a computation  in
\texttt{magma} by taking $k=10^5$.
\end{proof}


\section{Primitive components}\label{4}
In this section we see that, up to a number of exceptions of an order
of magnitude less that $n^{k-1},$ the $k$-partitions of $n$ with one
term coprime to the others, cannot
belong to a primitive proper subgroup of $ \Sym(n).$ We deduce this by
proving that the $k$-partitions in the set 
\begin{eqnarray*}
\mathfrak{T}_{k}(n)=\{[x_1,\dots, x_{k}]
\in\mathfrak{P}_{k}(n)&:&\exists\, i\in\{1,\dots,k\}\,  \hbox{with }\\\nonumber
&&  2\leq x_i< n/2,\,\gcd(x_i,\prod_{\substack{j=1\\j\neq i}}^{k}
x_j)=1\},
\end{eqnarray*}
defined in~\eqref{eq2}, cannot belong to a primitive subgroup, and by using an estimate for $|\mathfrak{T}_k(n)|$ in terms of the function $\hat{A}_k(n)$ defined in Notation \ref{e}.

\begin{lemma}\label{evenodd} Let $n,k\in\mathbb{N}$ with $n\geq k\geq 2.$ 
\begin{description}
\item[(i)] Let $n$ be odd and $G=\Sym(n)$, or let $n$ be even and
  $G=\Alt(n)$. If $H\lneqq G$ contains some
  $\mathfrak{p}\in\mathfrak{T}_{k}(n)$ with $k$ even, then $H$ is
  intransitive or imprimitive.
\item[(ii)] Let $n$ be even and $G=\Sym(n)$, or let $n$ be odd and
  $G=\Alt(n)$. If $H\lneqq G$ contains some
  $\mathfrak{p}\in\mathfrak{T}_{k}(n)$ with $k\geq 3$ odd, then $H$ is
  intransitive or imprimitive.
\end{description}
\end{lemma}
\begin{proof}
We prove part~(i) and~(ii) simultaneously. Let $H$ be a proper primitive
subgroup of $G$. Suppose that $H$ contains a permutation $\sigma$ of type $\mathfrak{p}(\sigma)=[x_1,\ldots,x_k]\in
\mathfrak{T}_k(n)$. Observe that the parity
conditions on $n$ and $k$ guarantee that $H\neq \Alt(n)$ when $G=\Sym(n)$ and
that $\mathfrak{T}_{k}(n)\subseteq \mathfrak{p}(\Alt(n))$ when
$G=\Alt(n)$. Thus in either case  $\Alt(n)\not\leq H$.

From the definition of $\mathfrak{T}_k(n)$, there exists $i\in
\{1,\ldots,k\}$ with $\gcd(x_i,x_j)=1$, for every
$j\neq i$ and $2\leq x_i< n/2.$ Write $m=\prod_{j\neq
  i}x_j$. Because of the coprimality condition,
$\sigma^m$ is a cycle of length $x_i$, where $2\leq x_i< n/2 .$ Thus,
by a celebrated theorem of Marggraf~\cite[Theorem~$13.5$]{WI}, the group $H$ is either $\Alt(n)$ or $\Sym(n)$, a contradiction.
\end{proof}

\begin{corollary}\label{Torder}Let $n,k\in\mathbb{N}$ with $n\geq k\geq 2.$ 
\begin{description}
\item[(i)] If $ k\geq 3$, we have
$$\hat{A}_k(n)-2(n-1)^{k-2}\leq |\mathfrak{T}_{k}(n)|\leq \hat{A}_k(n).$$
\item[(ii)] If $k=2$, we have $$|\mathfrak{T}_{2}(n)|=\hat{A}_2(n) -1.$$

\end{description}
\end{corollary}
\begin{proof} Let $n,k\in\mathbb{N}$ with $n\geq k\geq 2.$
Assume first $k=2.$ Recalling the description of the set $\mathfrak{A}_2(n)$ given in \eqref{mathfrak-A-2} and the definition  of $\mathfrak{T}_{2}(n)$, we get immediately that $\mathfrak{T}_2(n)=\mathfrak{A}_{2}(n)\setminus\{[1,n-1]\}$ and so (ii) follows.

Then let $ k\geq 3.$ Recalling that $\mathfrak{T}_{k}(n)\subseteq \mathfrak{A}_{k}(n),$ we get immediately the upper bound in (i).
Next set $\mathfrak{C}_k(n)=\mathfrak{A}_k(n)\setminus \mathfrak{T}_k(n):$ our aim is to find an upper estimate for $|\mathfrak{C}_k(n)|.$
We consider first the set $$\mathfrak{U}_k(n):=\{\mathfrak{p}\in \mathfrak{P}_k(n)\ :\ \hbox{there exists one term equal to }\ 1\}\subseteq \mathfrak{A}_{k}(n).$$
Then $|\mathfrak{U}_k(n)|$ is the number of $k$-partitions of $n$ of the type $[1,x_2\ldots,x_k]$
that is the number of $(k-1)$-partitions of
$n-1$, which is at most $(n-1)^{k-2}.$ In other words $|\mathfrak{U}_k(n)|\leq (n-1)^{k-2}.$

 Since $|\mathfrak{C}_k(n)|\leq |\mathfrak{C}_k(n)\setminus \mathfrak{U}_k(n)|+ |\mathfrak{U}_k(n)|,$ we can now reduce our problem to that of estimating the size of $\mathfrak{C}_k(n)\setminus \mathfrak{U}_k(n).$

 Let $\mathfrak{p}=[x_1,\ldots,x_k]\in \mathfrak{C}_k(n)\setminus \mathfrak{U}_k(n):$ then $x_i\geq 2$ for all $i=1,\dots,k$, $\mathfrak{p}$ has at least a term $x_j$ coprime with all the others but
 no term $x_i$ coprime with all the others and satisfying $2\leq x_i<n/2$. Thus $\mathfrak{p}$ contains a unique term $x_j\geq n/2$ coprime with the others. Moreover, removing $x_j$ from $\mathfrak{p}$ we obtain a partition of $(n-x_j)\leq n/2$ with $k-1$ parts. Thus we have $|\mathfrak{C}_k(n)\setminus \mathfrak{U}_k(n)|\leq (n/2)^{k-2}\leq (n-1)^{k-2}$.
\end{proof}

\begin{corollary}\label{Torder2} For any $k\in\mathbb{N}$ with $k\geq 2,$ there exists $n_k\in\mathbb{N}$ with $n_k\geq k$ such that $|\mathfrak{T}_{k}(n)|>0$ for all $n\in\mathbb{N}$ with $n\geq n_k.$ Moreover $|\mathfrak{T}_{k}(n)|$ has the same order of magnitude as $\hat{A}_k(n).$
\end{corollary}
\begin{proof} Let first $k=2.$ Then, by \eqref{mathfrak-A-2} and Corollary \ref{Torder}(ii), we get that $ |\mathfrak{T}_{2}(n)|=\hat{A}_2(n) -1= \varphi(n)/2-1.$ Thus $ |\mathfrak{T}_{2}(n)|>0$ for all $n\geq 7$ and, by Lemma \ref{o}, the order of magnitude of $ |\mathfrak{T}_{2}(n)|$  is the same as $\hat{A}_2(n)$ .

Let now $k\geq 3.$ By Corollary \ref{Torder}(i) we have \begin{equation}\label{Tord} \hat{A}_k(n)-2(n-1)^{k-2}\leq |\mathfrak{T}_{k}(n)|\leq \hat{A}_k(n).\end{equation} Since, by Theorem \ref{BLS}, $\hat{A}_k(n)$ has  order of magnitude $n^{k-1},$ there exists some positive constant $a_k\in \mathbb{R}$ such that $\hat{A}_k(n)\geq a_k n^{k-1}$ for all $ n\in\mathbb{N}$ with $n\geq k.$
Thus $|\mathfrak{T}_{k}(n)|\geq a_k n^{k-1}-2(n-1)^{k-2}$ and
 $\displaystyle{\lim_{n\rightarrow +\infty}a_k n^{k-1}-2(n-1)^{k-2}=+\infty}$ gives $\displaystyle{\lim_{n\rightarrow +\infty}|\mathfrak{T}_{k}(n)|=+\infty}.$  In particular there exists $n_k\geq k$ with $|\mathfrak{T}_{k}(n)|>0$ for all $n\geq n_k.$
Since $(n-1)^{k-2}=o(\hat{A}_k(n)),$ Lemma \ref{o} applied to \eqref{Tord} gives that $|\mathfrak{T}_{k}(n)|$  has the same order of magnitude as $\hat{A}_k(n).$
\end{proof}


\section{Imprimitive components}
In this section we look at a larger class of $k$-partitions of $n$, namely those
$[x_1,\ldots,x_k]$ with the
terms globally coprime, that is, $\gcd(x_1,\ldots,x_k)=1$. We show
that, for $k\in \{3,4\}$, the number of such partitions belonging to
an imprimitive subgroup of $\Sym(n)$ is of an order of magnitude
smaller than $n^{k-1}.$  In particular also
the $k$-partitions in $\mathfrak{T}_{k}(n)$ belonging to some
imprimitive subgroup will have an order of magnitude smaller than
$n^{k-1}.$

\begin{lemma}\label{lemma1} Let $n$ be a positive integer. The number
  of divisors of $n$ is at most $2\sqrt{n}$.
\end{lemma}
\begin{proof} The divisors of $n$ occur in pairs $\{b,n/b\}$ with
$b\leq\sqrt{n}$, (where we allow $\sqrt{n}$ paired with itself
when $n$ is a square).
\end{proof}

\begin{lemma}\label{lemma2}
Let $n$ and $k$ be positive integers, with $k\in \{3,4\}$ and $n\geq
k.$ Then the number of partitions $[x_1,\ldots,x_k]$ of $n$ with
$\gcd(x_1,\ldots,x_k)=1$ and belonging to some imprimitive permutation
group is at most $(k-1)n^{k-3/2}$.
\end{lemma}
\begin{proof}
Let $\mathfrak{p}=[x_1,\ldots,x_k]$ be a partition of $n$ with
$\gcd(x_1,\ldots,x_k)=1$ and let $\sigma=\sigma_1\cdots\sigma_k\in
\Sym(n)$ be a permutation of type $\mathfrak{p}$, where
$\sigma_1\cdots\sigma_k$ is a decomposition of $\sigma$ into pairwise disjoint
cycles of lengths $x_1,\ldots,x_k$.  Moreover, for each $i\in
\{1,\ldots,k\}$, denote by $X_i$ the support of $\sigma_i$. Suppose
that $\sigma$ belongs to some maximal imprimitive subgroup of
$\Sym(n)$, say to $W=\Sym(b)\wr\Sym(m)$, for some proper divisor $b$
of $n$ and $m=n/b.$

Let $\mathcal{B}$ be the system of imprimitivity for $W$. So,
$\mathcal{B}$ consists of $m$ blocks of size $b$. Let $i\in
\{1,\ldots,k\}$ and let $\mathcal{B}_i:=\{B\in \mathcal{B}\mid B\cap
X_i\neq \emptyset\}$. Observe that  $\langle\sigma_i\rangle$ acts
transitively on $\mathcal{B}_i$ and that the action of $\sigma_i$ on
$\mathcal{B}_i$ is equivalent to the action of $\sigma$ on
$\mathcal{B}_i$. If follows that $|B\cap X_i|=|B'\cap X_i|$, for every
two elements $B,B'\in\mathcal{B}_i$, and that  $\mathcal{B}_i\cap
\mathcal{B}_j=\emptyset$ or $\mathcal{B}_i=\mathcal{B}_j$, for every
two elements $i,j\in \{1,\ldots,k\}$.

Assume that $\mathcal{B}_i=\mathcal{B}$, for some $i\in
\{1,\ldots,k\}$. Then $m$ divides $x_i$ because
$\langle\sigma_i\rangle$ acts transitively on $\mathcal{B}$. Moreover,
from the previous paragraph, we have $\mathcal{B}=\mathcal{B}_j$, for
every $j$. It follows that $m$ divides $x_j$, for every $j$,
contradicting $\gcd(x_1,\ldots,x_k)=1$.

Assume that $\mathcal{B}_1,\ldots,\mathcal{B}_k$ is a partition of
$\mathcal{B}$, that is, $\mathcal{B}_i\cap \mathcal{B}_j=\emptyset$
for distinct $i, j\in\{1,\ldots,k\}$. Now, the
set $X_i$ is a union of blocks of imprimitivity and hence $b$ divides
$x_i$, for every $i$, again contradicting  $\gcd(x_1,\ldots,x_k)=1$.

Relabeling the index set $\{1,\ldots,k\}$ if necessary, from the
previous two paragraphs we may assume that
$\mathcal{B}_1=\mathcal{B}_2$ and that $\mathcal{B}_3\cap
\mathcal{B}_1=\mathcal{B}_3\cap\mathcal{B}_2=\emptyset$.

We now use the condition $k\in \{3,4\}$. Assume first that $k=3$.
Then $\mathcal{B}=\mathcal{B}_1\cup\mathcal{B}_3$, and it follows
that $X_3$ is the union of the blocks from $\mathcal{B}_3$, and hence
$b$ divides $x_3$.  Now fix $B\in\mathcal{B}_1$ and observe that
$x_i=|\mathcal{B}_1||B\cap X_i|$ for $i\in \{1,2\}$, because
$\mathcal{B}_1=\mathcal{B}_2$ and $\mathcal{B}_1\cap
\mathcal{B}_3=\emptyset$. Since
$|\mathcal{B}_1|=|\mathcal{B}|-|\mathcal{B}_3|=m-x_3/b=(n-x_3)/b$. It
follows that
$$
x_1=\frac{n-x_3}{b}|B\cap X_1|
$$
and
$$
x_2=\frac{n-x_3}{b}(b-|B\cap X_1|),
$$
because we must have $b=|B\cap X_1|+|B\cap X_2|$. Hence the partition
$[x_1,x_2,x_3]$ is determined by $x_3$ (which is divisible by $b$) and
by $|B\cap X_1|$ (which is at most $b$). Hence $W$ contains at most
$(n/b)\cdot b=n$ partitions of type $[x_1,x_2,x_3]$ with
$\gcd(x_1,x_2,x_3)=1$. As we have at most $2n^{1/2}$ divisors $b$ of $n$
by Lemma~\ref{lemma1}, at most $2n^{1/2}.n=(k-1)n^{k-3/2}$ of these
types belong to a given imprimitive group, proving the result.

Assume now that $k=4$, and recall that $\mathcal{B}_1=\mathcal{B}_2$
and $\mathcal{B}_1\cap\mathcal{B}_3=\emptyset$. Suppose first that
$\mathcal{B}_4=\mathcal{B}_3$. Then
$\mathcal{B}=\mathcal{B}_1\cup\mathcal{B}_3$ with
$\mathcal{B}_1\cap\mathcal{B}_3=\emptyset$. Let $B\in \mathcal{B}_1$
and $C\in \mathcal{B}_3$. Clearly,
$x_i=|\mathcal{B}_1||B\cap X_i|$ for $i\in \{1,2\}$, and
$x_i=|\mathcal{B}_3||C\cap X_i|$ for $i\in \{3,4\}$. Also, $B=(B\cap
X_1)\cup (B\cap X_2)$ and $C=(C\cap X_3)\cup (C\cap X_4)$. Thus
\begin{align*}
x_1&=|\mathcal{B}_1||B\cap X_1|,& x_2&=|\mathcal{B}_1|(b-|B\cap X_1|),\\
x_3&=(m-|\mathcal{B}_1|)|C\cap X_3|,& x_4&=(m-|\mathcal{B}_1|)(b-|C\cap X_3|).\\
\end{align*}
It follows that the partition $[x_1,x_2,x_3,x_4]$ is determined by
$|B\cap X_1|$ and $|C\cap X_3|$ (which are both less than $b$) and by
$|\mathcal{B}_1|$ (which is less than $m=n/b$). Hence $W$ contains less
than $b\cdot b\cdot (n/b)=bn$ partitions of type $[x_1,x_2,x_3,x_4]$
such that  $\gcd(x_1,x_2,x_3,x_4)=1$ and having this prescribed configuration.

Next suppose that $\mathcal{B}_4=\mathcal{B}_1$. Then again
$\mathcal{B}=\mathcal{B}_1\cup\mathcal{B}_3$ with
$\mathcal{B}_1=\mathcal{B}_2=\mathcal{B}_4$ and
$\mathcal{B}_1\cap\mathcal{B}_3=\emptyset$. Let $B\in
\mathcal{B}_1$. We have $x_i=|\mathcal{B}_1||B\cap X_i|$ for $i\in
\{1,2,4\}$ and $B=(B\cap X_1)\cup (B\cap X_2)\cup (B\cap X_4)$. Thus
\begin{align*}
x_1&=|\mathcal{B}_1||B\cap X_1|,& x_2&=|\mathcal{B}_1||B\cap X_2|,\\
x_4&=|\mathcal{B}_1|(b-|B\cap X_1|-|B\cap X_2|),&x_3&=n-x_1-x_2-x_4.\\
\end{align*}
Hence the partition $[x_1,x_2,x_3,x_4]$ is determined by
$|B\cap X_1|$ and $|B\cap X_2|$ (which are both less than $b$) and by
$|\mathcal{B}_1|$ (which is less than $m=n/b$). Hence again $W$ contains less
than $b\cdot b\cdot (n/b)=bn$ partitions of type $[x_1,x_2,x_3,x_4]$
such that $\gcd(x_1,x_2,x_3,x_4)=1$ and having this prescribed configuration.

Finally suppose that $\mathcal{B}_4\neq
\mathcal{B}_1,\mathcal{B}_3$. Thus $\mathcal{B}_4\cap
\mathcal{B}_1=\emptyset$, $\mathcal{B}_4\cap\mathcal{B}_3=\emptyset$
and $\mathcal{B}=\mathcal{B}_1\cup\mathcal{B}_3\cup\mathcal{B}_4$. Let
$B\in \mathcal{B}_1$ and $C\in \mathcal{B}_3$. From the definition of $\mathcal{B}_i$, we have
$X_1\cap C=X_2\cap C=X_4\cap C=\emptyset$ and hence $C\subseteq
X_3$. Thus $X_3$ is the union of the blocks in
$\mathcal{B}_3$. Similarly, $X_4$ is the union of the blocks in
$\mathcal{B}_4$. In particular, $b$ divides $x_3$ and $x_4$. Moreover,
for $B\in \mathcal{B}_1$, we have $x_1=|\mathcal{B}_1||B\cap
X_1|$. Note that
$|\mathcal{B}_1|=|\mathcal{B}|-|\mathcal{B}_3|-|\mathcal{B}_4|$. Thus
\begin{align*}
x_1&=|\mathcal{B}_1||B\cap X_1|,& x_2&=|\mathcal{B}_1|(b-|B\cap X_1|),\\
x_3&= |\mathcal{B}_3|b,& x_4&=n-x_1-x_2-x_3.\\
\end{align*}
Hence the partition $[x_1,x_2,x_3,x_4]$ is determined by
$|\mathcal{B}_1|$ and  $|\mathcal{B}_3|$ (which are both less than
$n/b$) and by $|B\cap X_1|$ (which is less than $b$). Hence in this case $W$ contains
less than $(n/b)\cdot (n/b)\cdot b=n^2/b$ partitions of type
$[x_1,x_2,x_3,x_4]$ with $\gcd(x_1,x_2,x_3,x_4)=1$ and having this
prescribed configuration.

Summing up, for $k=4$, the group $W$ contains at most $bn+bn+n^2/b$
partitions of type $[x_1,x_2,x_3,x_4]$ with
$\gcd(x_1,x_2,x_3,x_4)=1$. It is an immediate computation (using that
$2\leq b\leq n/2$) to see that $2bn+n^2/b\leq n^2+2n\leq 3n^2/2$.
Finally, applying  Lemma~\ref{lemma1}, at most $3n^{5/2}=(k-1)n^{k-3/2}$
of these types belong to a given imprimitive group.
\end{proof}

\begin{remark}\label{rem1}{\rm We point out that the hypothesis $\gcd(x_1,\ldots,x_k)=1$ in Lemma~\ref{lemma2} is essential.
In fact, if $n$ is  divisible by a  prime $p$, then all the
partitions $[x_1'p,\ldots,x_k'p]$ of $n$ belong to
$\Sym(p)\wr\Sym(n/p)$ and to $\Sym(n/p)\wr\Sym(p)$. Now, the number of
$k$-partitions of $n$ having all the entries divisible by $p$ is
exactly the number of $k$-partitions of $n/p$ and, by~\eqref{K} and~\eqref{part-comp}, this number is a polynomial of degree $k-1$ in $n/p$. Thus
if $p$ is small compared to $n$  it behaves like a polynomial of
degree $k-1$ in $n$ and therefore its order of magnitude is greater
than $n^{k-3/2}.$}
\end{remark}


\section{Intransitive components }
In this section we focus on the problem of understanding how many
partitions  belong
to a fixed intransitive subgroup: this will give the final tool for
the proof of Theorems~\ref{thm3} and~\ref{main1}.
We start with an elementary counting lemma, the proof of which can be found
in~\cite[page~$81$]{Andrews}.

\begin{lemma}\label{lemma5}
Let $n\in\mathbb{N}.$ The number of $2$-partitions of
$n$ is $\lfloor n/2\rfloor$ and the number of
$3$-partitions of $n$ is
\[|\mathfrak{P}_3(n)|=\left\{
\begin{array}{lcl}
\frac{(n-1)(n-2)}{12}+\frac{1}{2}\lfloor\frac{n-1}{2}\rfloor&&\textrm{if
}\,\gcd(n,3)=1,\,\textrm{and}\\
\frac{(n-1)(n-2)}{12}+\frac{1}{2}\lfloor\frac{n-1}{2}\rfloor+\frac{1}{3}&&\textrm{if
}\,3\mid n\\
\end{array}\right.
\]
\end{lemma}

\begin{lemma}\label{lemma6}
Let $n\in\mathbb{N}.$
Then the numbers of $\,3$-partitions of $n$ and $\,4$-partitions of $n$ belonging to
an intransitive subgroup 
of $\,\Sym(n)$ is at most $n/2$ and $(5n^2+23)/48$ respectively.
\end{lemma}
\begin{proof} Clearly it is enough to run the proof only for the maximal intransitive subgroup of $\Sym(n)$, that is for
$H=\Sym(a)\times \Sym(n-a)$, with $1\leq a<n/2.$  Let $\mathfrak{p}=[x_1,\ldots,x_k]$
be a partition of $n$ with $k$
terms, where $k\in \{3,4\}$. Clearly, $\mathfrak{p}$ is contained in $H$ if and only if the sum of some of the terms in
$\mathfrak{p}$ is $a$. Suppose that $a=x_i$, for some $i\in
\{1,\ldots,k\}$. Then the number of possibilities for $\mathfrak{p}$
is exactly the number of partitions of $n-a$ with $k-1$ terms and this
number is computed in Lemma~\ref{lemma5}, replacing $n$ by $n-a$.

Suppose next that $a=x_i+x_j$ for some $i,j\in \{1,\ldots,k\}$ with
$i\neq j$. Observe that $a\neq n-a$ because $a<n/2$. The number of
possibilities for $\mathfrak{p}$ is the number of $2$-partitions of
$a$  times the number of $(k-2)$-partitions of $n-a$. From
Lemma~\ref{lemma5}, this number is $\lfloor a/2\rfloor$ if $k=3$ and
$\lfloor a/2\rfloor\lfloor (n-a)/2\rfloor$ if $k=4$.

Finally suppose that $k=4$ and that $x_i=n-a$ for some $i\in
\{1,\ldots,k\}$. Then the number of possibilities for $\mathfrak{p}$
is exactly the number of  $3$-partitions of $a$  and this number is
computed in Lemma~\ref{lemma5}, replacing $n$ by $a$.

By adding the previous contributions and by recalling that $\lfloor
x\rfloor\leq x$ for $x\geq 0$, we see that the number of
$3$-partitions contained in $H$ is at most $a/2+(n-a)/2=n/2$, and that
the number of $4$-partitions contained in $H$ is at most
\begin{equation*}
\begin{aligned}
&\frac{(n-a-1)(n-a-2)}{12}+\frac{n-a-1}{4}+\frac{2}{3}+\frac{a(n-a)}{4}\\&
+\frac{(a-1)(a-2)}{12}+\frac{a-1}{4}
=\frac{n^2+an-a^2+6}{12}.
\end{aligned}
\end{equation*}
Since $f(a):=\frac{n^2+an-a^2+6}{12}$ is increasing with $a\leq (n-1)/2
$, the result follows by computing $f((n-1)/2)$.
\end{proof}


\section{ Proof of the main results}

Before proceeding with the proof of Theorem~\ref{main1} it is
convenient to introduce some notation and some terminology.
If $\mathfrak{p}$
is a $k$-partition, we say that $\mathfrak{p}$ \textit{belongs only to
intransitive subgroups} of $G$ if, whenever $H\lneqq G$ and $\mathfrak{p}\in
H$, then $H$ is intransitive.

\begin{notation}\label{tildee} For each $k\in\{3,4\}$, $n\in\mathbb{N}$ with $n\geq k$ and $G=\Sym(n)$ or $G=\Alt(n)$, we set
\begin{equation*}
\begin{aligned}
& \widetilde{A}_k(n,G):=|\{\mathfrak{p}\in \mathfrak{P}_{k}(n): \mathfrak{p}\  \hbox{belongs only to intransitive subgroups of} \ G \}|,\\
 &\tilde{e}_3(n):=e_3(n)-2n^{3/2}-2(n-1),\qquad \tilde{e}_4(n):=e_4(n)-3n^{5/2}-2(n-1)^2.\\
 \end{aligned}
 \end{equation*}
 \end{notation}
Observe that it is possible to have $\widetilde{A}_k(n,G)=0$. For example,   $\widetilde{A}_3(n,G)=0$ when $n$ is even and $G=\Alt(n)$, and when  $n$ is odd and $G=\Sym(n)$. Similarly $\tilde{A}_4(n,G)=0$ when $n$ is even and $G=\Sym(n)$, and when  $n$ is odd and $G=\Alt(n)$. We show that, in all other cases, the quantity $\widetilde{A}_k(n,G)$ is positive for $n$ sufficiently large.
\begin{lemma}\label{part-int}
\begin{description}
\item[(i)] Let $n\in\mathbb{N},$ with $n\geq 3.$ If $n$ is even and $G=\Sym(n)$ or if $n$ is odd and
  $G=\Alt(n)$, then
$$
\widetilde{A}_3(n,G)\geq \frac{C_2\,f_2(n)}{12}n^{2}+\tilde{e}_3(n),
$$
and $\tilde{e}_3(n)=o(n^{2})$.
In particular there exists $n_3\in\mathbb{N}$ such that $\widetilde{A}_3(n,G)>0$ for all $n\geq n_3$ respecting the parity conditions, and $\widetilde{A}_3(n,G)$ has order of magnitude $n^2.$
\item[(ii)] Let $n\in\mathbb{N},$ with $n\geq 4.$ If $n$ is odd and $G=\Sym(n)$ or if $n$ is even and
  $G=\Alt(n)$, then
$$
\widetilde{A}_4(n,G)\geq  \frac{
   C_3\,f_3(n)}{144}n^{3}+\tilde{e}_4(n),
$$
and $\tilde{e}_4(n)=o(n^{3})$.
In particular there exists $n_4\in\mathbb{N}$ such that $\widetilde{A}_4(n,G)>0$ for all $n\geq n_4$ respecting the parity conditions, and $\widetilde{A}_4(n,G)$ has order of magnitude $n^3.$

\end{description}
\end{lemma}

\begin{proof} We give full proof details for Part~(i). Part~(ii) follows by a similar argument applied to the $4$-partitions in $\mathfrak{T}_{4}(n).$

Let $n\in\mathbb{N},$ with $n\geq 3,$ with $n$  even and $G=\Sym(n)$ or  $n$  odd and
  $G=\Alt(n)$.
By Lemma~\ref{evenodd}, the elements of
  $\mathfrak{T}_3(n)$ belongs only to proper subgroups of $G$ that are intransitive or imprimitive and,
  by Lemma~\ref{lemma2}, at most $2n^{3/2}$ belong to some imprimitive
  subgroup of $G$.
So, by Corollary \ref{Torder} we have
$$
\widetilde{A}_3(n,G)\geq |\mathfrak{T}_{3}(n)|-2n^{3/2}\geq
\hat{A}_3(n)-2(n-1)-2n^{3/2}.
$$
Now the inequality follows from Theorem~\ref{coprimepart} and Notation~\ref{tildee}. Since, by Lemma~\ref{f_23}, $f_2(n)>1$ and
$\displaystyle{\lim_{n\rightarrow +\infty}\frac{C_2}{12}n^{2}+\tilde{e}_3(n)=+\infty},$ 
there exists $n_3\in\mathbb{N}$ such that $\widetilde{A}_3(n,G)$ is positive in the unbounded domain of the $n\in\mathbb{N}$ with $n\geq n_3,$ respecting the parity conditions, while $\widetilde{A}_3(n,G)=0$ if the parity conditions are not satisfied.

Finally since $\widetilde{A}_3(n,G)\leq |\mathfrak{P}_3(n)|$ and $|\mathfrak{P}_3(n)|$ has order of magnitude $n^2$, it follows from Lemma~\ref{o}  that $\widetilde{A}_3(n,G)$ has order of magnitude $n^2.$
\end{proof}

The quantities $\widetilde{A}_k(n,G)$ play a crucial role in finding lower
bounds for $\kappa(G)$. Before seeing that, we examine when the function $\kappa$ is positive.
\begin{proposition}\label{kappasym} Let $n\in\mathbb{N}$. Then:
\begin{description}
\item[(i)] $\kappa(\Alt(n))\geq 2$ for all $n\geq 3.$
\item[(ii)] $\kappa(\Sym(n))\geq 2$ for all $n\geq 2,$ with $n\neq 6$, while $\kappa(\Sym(6))=0.$
\end{description}

\end{proposition}
\begin{proof} For $n\geq 5$, part (i) follows from \cite[Proposition 7.4]{GM}, and
it is easy to check that  $\kappa(\Alt(3))=3$ and  $\kappa(\Alt(4))=2.$


Now we consider $\Sym(n).$ If $n$ is odd we note first that $\kappa(\Sym(3))= 2$
(consider types $[3]$ and $[1,2]$), and for $n\geq5$ we
let $\sigma$  be any permutation of type $[n]$,
$\tau$ any permutation of type $[2,n-2]$,  and $U=\langle \sigma,\tau\rangle$. Since
$\tau$  is an odd permutation, we have that $U\nleq Alt(n)$. Now $U$ is transitive since it contains
$\sigma$, and $U$ cannot be imprimitive since it contains $\tau$ and  $\gcd(2,n-2)=1$.
Thus $U$ is primitive and, noting that $[2,n-2]\in\mathfrak{T}_2(n)$, it follows from
Lemma \ref{evenodd}(i) that we must have $U=\Sym(n)$.

Next let $n$ be even with $n\geq 10$.  Let $\sigma$  be any permutation of type $[1,2,n-3]$,
$\tau$ any permutation of type $[5,n-5]$, and suppose that $U:=\langle \sigma,\tau\rangle\ne\Sym(n).$
Now $[1,2,n-3]\in\mathfrak{T}_3(n)$, and it follows from Lemma \ref{evenodd}(ii) that $U$ is
not primitive. Also $U$ is transitive since the sum of at most two terms in $[1,2,n-3]$
cannot be equal to $5.$ Hence $U$ is imprimitive, say $U\leq \Sym(b)\wr\Sym(n/b)$,
for some proper divisor $b$ of $n.$  Since $[5,n-5]\in U,$ it follows that $5$ divides
$n$ and $b\in\{5, n/5\}$; on the other hand, since also $[1,2,n-3]\in U$ and
$\gcd(1,2, n-3)=1$ we must have $b=3$ which implies that $n=15$, a contradiction since $n$ is even.
Thus we conclude that $U=\Sym(n).$

The remaining cases are $n=2,4,6,8$. We see that $\kappa(\Sym(2))=2$ , $\kappa(\Sym(4))\geq 2 $ and $\kappa(\Sym(8))\geq 2,$
on considering the pairs of types $[1], [2]$ for $n=2,$ $[1,3], [4]$ for $n=4$ and $[8], [3,5]$ for $n=8.$
%
%
Finally let $G=\Sym(6).$ We show that for each pair of types $\mathfrak{p}_1,\mathfrak{p}_2
\in \mathfrak{P}(G),$ there exist $\sigma\in G$ of type $\mathfrak{p}_1$ and $\tau\in G$ of type
$\mathfrak{p}_2$ with $U=\langle \sigma,\tau\rangle<G.$ This is obvious if $1$ is a term in
both $\mathfrak{p}_1$ and $\mathfrak{p}_2$, so we may assume that $\sigma$ has no fixed point.
If $\mathfrak{p}_1=[6]$ or $[2,2,2]$, then $\mathfrak{p}_1$ belongs to both $W_1=\Sym(3)\wr\Sym(2)$ and the
2-transitive subgroup $S=\mathrm {PGL}(2,5)$. The only types of permutations in $G$ which do not
belong to $W_1$ are $[1,5],\ [1,1,4]$, and these types belong to $S$.
If $\mathfrak{p}_1=[2,4]$ or $[3,3]$, then $\mathfrak{p}_1$ belongs to all of $W_1$,
$W_2=\Sym(2)\wr\Sym(3)$ and $\Alt(6)$, and we note that $[1,1,4]$ belongs to $W_2$, and $[1,5]$
belongs to $\Alt(6)$.

\end{proof}

\begin{lemma}\label{newkappa}
Let $n\in\mathbb{N}$ with $n\geq 3$ and $G=\Sym(n)$ or $G=\Alt(n)$, with $(n,G)\neq (6,\Sym(6))$. Then  \[\kappa(G)\geq\max \left\{\displaystyle{2,\, \frac{2}{3n}}\ \widetilde{A}_3(n,G),\displaystyle{\frac{24}{25n^2}}\ \widetilde{A}_4(n,G)\right\}.
\]

\end{lemma}

We recall that, for each choice of $n\in\mathbb{N}$ and $G,$ one of $\widetilde{A}_3(n,G)$ and $\widetilde{A}_4(n,G)$ is $0$, and  by Lemma~\ref{part-int}, provided $n$ is sufficiently large, the other is positive.

\begin{proof}

Let $n\in\mathbb{N}$ with $n\geq 3$ and $G=\Sym(n)$ or $G=\Alt(n)$ with $(n,G)\neq (6,\Sym(6))$.
Then, by Proposition \ref{kappasym}, we know that $\kappa(G)\geq 2.$ Hence, if
$M(G,n):= \max\{ 2\widetilde{A}_3(n,G)/(3n),\ 24\widetilde{A}_4(n,G)/(25n^2)\}\leq 2,$ the result follows.
So we may assume that $M(G,n)>2$. Let $k\in\{3,4\}$ be (the unique integer) such that $\widetilde{A}_k(n,G)>0.$ In order to show that $\kappa(G)\geq M(G,n),$
we  construct an auxiliary bipartite graph $\Sigma$.
The vertices of $\Sigma$ are the elements  of $$\Sigma_1:=\{\mathfrak{p}\in \mathfrak{P}_k(n):
\mathfrak{p} \textrm{ belongs only to intransitive subgroups of }G\}$$ and of
$$\Sigma_2:=\mathfrak{P}_2(n)=\{[a,n-a]: 1\leq a< n/2\}.$$
Observe that $|\Sigma_1|=\widetilde{A}_k(n,G)>0.$
Define a vertex $\mathfrak{p}\in
\Sigma_1$ to be adjacent
to $[a,n-a]\in \Sigma_2$ if 
$\mathfrak{p}$ belongs to $\Sym(a)\times \Sym(n-a)$, 
 that is, $a$ is the sum of some of the terms of
$\mathfrak{p}$. If $k=3$, then
every element $[x,y,z]$ of $\Sigma_1$ has at most $3$ neighbours in
$\Sigma_2$, and  if
$k=4$, then  every element $[x,y,z,t]$ of $\Sigma_1$ has at most $10$
neighbours in $\Sigma_2$. 

Furthermore, by
Lemma~\ref{lemma6}, every element of $\Sigma_2$ has at most $n/2$
neighbours in $\Sigma_1$ if $k=3$, and at most $(5n^2+23)/48$
neighbours in $\Sigma_1$ if $k=4$.
Hence there are at most  $3 (\frac{n}{2} -1)+1<\frac{3n}{2}$
vertices at distance less than or equal to $2$
from each element of $\Sigma_1$ if $k=3$, and at most
$10\,(\frac{5n^2+23}{48}-1)+1< \frac{25n^2}{24}$ such vertices if $k=4$.
Set $\ell:=\lceil \frac{2|\Sigma_1|}{3n}\rceil$ if $k=3$
and $\ell:=\lceil\frac{ 24|\Sigma_1|}{25n^2}\rceil$ if $k=4$. Since $\ell\geq M(G,n)> 2,$ it is enough to prove that
$\kappa(G)\geq \ell.$
By definition of $\ell$, $\Sigma_1$ contains a subset $\{\mathfrak{p}_1,\ldots,\mathfrak{p}_\ell\}$
of $\ell > 2$ vertices which are pairwise at distance greater than $2$. We show that for every $i,j\in
\{1,\ldots,\ell\}$ with $i\neq j$, and for every $\sigma,\tau\in G$
with $\mathfrak{p}(\sigma)=\mathfrak{p}_i$ and $\mathfrak{p}(\tau)=\mathfrak{p}_j$, we
have $G=\langle \sigma,\tau\rangle$. We argue by contradiction and assume that
$U=\langle \sigma,\tau\rangle$ is a proper subgroup of $G$. By the
definition of $\Sigma_1$, $U$ is intransitive and we have
$U\leq \Sym(a)\times \Sym(n-a)$ for some positive integer $a<n/2$. This implies that
$\mathfrak{p}_i$ and $\mathfrak{p}_j$ are both adjacent to $[a,n-a]\in \Sigma_1$, 
and so at distance $0$ or $2$, contradicting the definition of the set $\{\mathfrak{p}_i\ :\ i=1,\dots,\ell \}.$
\end{proof}


\begin{proof}[Proof of Theorem~\ref{main1}]
Let assume $n\geq 7$ to avoid the critical case $n=6$ of Lemma ~\ref{newkappa}.
First let $n$ be even and $G=\Sym(n)$, or let $n$ be odd and $G=\Alt(n)$.
By Lemma~\ref{newkappa}, $\kappa(G)\geq \frac{2}{3n}\,\widetilde{A}_3(n,G)$,
and hence by Lemma~\ref{part-int},
\begin{equation}\label{gamma2}
\kappa(G) \geq \frac{C_2 f_2(n)}{18}n +\frac{2\tilde{e}_3(n)}{3n}
\end{equation}
with $\tilde{e}_3(n)=o(n^{2})$.
Let $\varepsilon\in\mathbb{R},$ with $\varepsilon>0$. As $\tilde{e}_3(n)=o(n^2)$, there exists $n_\varepsilon\in \mathbb{N}$ with $n_\varepsilon\geq 7$ such that
$|2\tilde{e}(n)/3n^2|<\varepsilon$, for every $n\geq n_\varepsilon$.
In particular, using Lemma~\ref{f_23} and inequality~\eqref{gamma2} we get
\begin{eqnarray*}
\frac{\kappa(\Sym(n))}{n}&\geq&
\frac{C_2}{12}-\varepsilon\qquad\textrm{ if }n\geq n_\varepsilon \textrm{ is even},\\
\frac{\kappa(\Alt(n))}{n}&\geq&\frac{C_2}{18}-\varepsilon\qquad\textrm{
  if }n\geq n_\varepsilon \textrm{ is odd}.
\end{eqnarray*}
Now simply recall that $C_2=\alpha$.

Now let $n$ be even and $G=\Alt(n)$, or let $n$ be odd and
$G=\Sym(n)$.
By Lemma~\ref{newkappa}, $\kappa(G)\geq \frac{24}{25n^2}\,\widetilde{A}_4(n,G)$,
and hence by Lemma~\ref{part-int},
\begin{equation}\label{gamma3}
\kappa(G)\geq \frac{C_3 f_3(n)}{150}n
+\frac{24\tilde{e}_4(n)}{25n^2}
\end{equation}
with $\tilde{e}_4(n)=o(n^{3})$.
Let $\varepsilon\in\mathbb{R},$ with $\varepsilon>0$. As $\tilde{e}_4(n)=o(n^3)$, there exists $n_\varepsilon\in \mathbb{N}$ with $n_\varepsilon\geq 7$ such that
$|24\tilde{e}_4(n)/25n^3|<\varepsilon$, for every $n\geq n_\varepsilon$.
In particular, using Lemma~\ref{f_23},  Notation~\ref{constants} and inequality~\eqref{gamma3} we get
%
\begin{eqnarray*}
\frac{\kappa(\Sym(n))}{n}&\geq&
\frac{C_3}{150}\frac{5\alpha'}{4}-\varepsilon=\frac{\beta}{120}-\varepsilon\qquad \textrm{ if }n\geq n_\varepsilon  \textrm{ is odd},\\
\frac{\kappa(\Alt(n))}{n}&\geq&\frac{C_3}{150}\alpha'-\varepsilon=\frac{\beta}{150}-\varepsilon\qquad \textrm{ if }n\geq n_\varepsilon  \textrm{ is even}.
\end{eqnarray*}
\end{proof}

\begin{proof}[Proof of Theorem~\ref{thm3}]
Write $c_0:=0.001911$. A direct computation using Lemma~\ref{f_23}~(iii) shows
that
$$
\min\left\{\frac{C_2}{12},\frac{C_2}{18}, \frac{\beta}{120},
\frac{\beta}{150}\right\}>c_0.
$$
Let $n\in\mathbb{N}$ and $G=\Sym(n)$ with $n\geq 3, n\neq 6$ or $G=\Alt(n)$ with $n\geq 4:$ these limitations for $n$ guarantee that both the functions $\kappa(G)$ and $\gamma(G)$ exist and are positive.
Fix $\bar{\varepsilon}\in\mathbb{R}$ with
$0<\bar{\varepsilon}<c_0$. By Theorem \ref{main1},
there exists  $n_{\bar{\varepsilon}}\in \mathbb{N}$ with $n_\varepsilon\geq 7$, such that
$\kappa(G)/n>c_0-\bar{\varepsilon}$, for each $n\geq
n_{\bar{\varepsilon}}.$

Define
\begin{eqnarray*}
c_1&:=&\min\left\{\frac{\kappa(\Sym(n))}{n}\ :\ \  3\leq n< n_{\bar{\varepsilon}},\,n\neq 6\right\}\\
c_2&:=&\min\left\{\frac{\kappa(\Alt(n))}{n}\ :\ \  4\leq n< n_{\bar{\varepsilon}}\right\}
\end{eqnarray*}
and observe that $c_1,c_2$ are well-defined and positive because, by Proposition \ref{kappasym}, they are minima of  finite sets of positive real numbers. Now let
$c:=\min\left\{c_0-\bar{\varepsilon},\,c_1,\,c_2\right\}.$ Then
 $0<c<0.001911$ and $\kappa(G)\geq cn$, for every $n\geq 3$ with $n\neq 6$
if $G=\Sym(n)$, and every $n\geq 4$ if $G=\Alt(n)$.

Finally recall the relation \eqref{kappa-gamma} and apply \cite[Theorem A and Table 3]{BP} for the upper bound.
\end{proof}

\begin{remark}\label{computational} {\rm

 Using some software like \texttt{Magma}~\cite{magma}, one can obtain explicitly,
 given a fixed $\varepsilon$,
    the positive integers $n_\varepsilon$ in Theorem
    \ref{main1}. The reason is that we obtained in \eqref{gamma2} and
    \eqref{gamma3}, explicit lower bounds for
    $\displaystyle{\frac{\kappa(G)}{n}}$. Here we consider the case that $n$ is even
    and $G=\Sym(n)$ (the other cases are similar). From~\eqref{gamma2}
    and from the definition of
    $\tilde{e}_3(n)$ in Notation~\ref{tildee}, we have
\begin{eqnarray*}
\frac{\kappa(G)}{n}&\geq&
\frac{C_2f_2(n)}{18}+\frac{2\tilde{e}_3(n)}{3n^2}=\frac{C_2f_2(n)}{18}+\frac{2e_3(n)-4n^{2/3}-4n+4}{3n^2}.
\end{eqnarray*}
Now recalling that $f_2(n)\geq 3/2$ and that $C_2>0.32263$ by
Lemma~\ref{f_23} and recalling the
definition of $e_3(n)$ in Notation~\ref{constants}, we get
\begin{eqnarray*}
\frac{\kappa(G)}{n}&\geq&\frac{C_2}{12}-\frac{(2+e)(e^2\log
  n)^2}{18n\sqrt{\pi}}-\frac{4n^{2/3}+4n-4}{3n^2}\\
&\geq& 0.026885-\frac{(2+e)(e^2\log
  n)^2}{18n\sqrt{\pi}}-\frac{4n^{2/3}+4n-4}{3n^2}.
\end{eqnarray*}
The function  on the right hand side of this inequality is explicit and a
computation shows that, for $n\geq 792000$, this function is
$>0.025$. Therefore if $n\geq 792000$ and $n$ is even, then
$\kappa(\Sym(n))\geq 0.025\, n$.}

\end{remark}


\thebibliography{10}

\bibitem{Andrews}G.~E.~Andrews, \textit{The Theory of Partitions}, Encyclopedia of
Mathematics and its Applications, Vol. 2,  Addison-Wesley Publishing
Co., Reading, MA-London-Amsterdam, 1976.

\bibitem{Bl}S.~R. Blackburn, Sets of permutations that generate the symmetric
 group pairwise, \textit{J. Combinatorial Theory, Series A} \textbf{113} (2006), 1572-1581.

\bibitem{magma}W.~Bosma, J.~Cannon, C.~Playoust, The Magma algebra
  system. I. The user language, \textit{J. Symbolic Comput.}
  \textbf{24} (1997), 235--265.

\bibitem{BM}J.~R.~Britnell, A.~Mar\'oti, Normal coverings of linear groups, available at http://arxiv.org/pdf/1206.4279.pdf.

\bibitem{BLS}D.~Bubboloni, F.~Luca, P.~Spiga, Compositions of $n$
  satisfying some coprimality conditions, \textit{Journal of Number
    Theory} \textbf{132} (2012), 2922--2946.

\bibitem{BP} D.~ Bubboloni, C.~ E.~ Praeger, Normal coverings of
  finite symmetric and alternating groups, \textit{Journal of
    Combinatorial Theory} Series A {\bf 118} (2011), 2000--2024.







\bibitem{Gould}H.~W.~Gould, Binomial coefficients, the bracket
  function, and compositions with relatively prime summands, \textit{Fibonacci
Quart.} \textbf{2} (1964), 241--260.

\bibitem{GM} R. Guralnick, G. Malle, Simple groups admit Beauville structures,  \textit{J. London Math. Soc.}(2) \textbf{85} (2012), 694--721.


\bibitem{M}
A. Mar\'oti, Covering the symmetric groups with proper subgroups, \textit{J. Combin. Theory,
Series A} \textbf{110} (2005), 97--111.

\bibitem{WI} H.~Wielandt, \emph{Finite Permutation Groups}, Academic
  Press, New York (1964).
\end{document}